\documentclass[twoside, 12pt]{article}

\usepackage{amsmath}
\usepackage{amsthm}
\usepackage{amsfonts}
\newcommand{\g}{\ensuremath{\mathfrak g}}
\newcommand{\go}{\ensuremath{{\mathfrak g}_0}}
\newcommand{\h}{\ensuremath{\mathfrak h}}
\newcommand{\kk}{\ensuremath{\mathfrak k}}
\newcommand{\kko}{\ensuremath{{\mathfrak k}_0}}
\newcommand{\so}{\ensuremath{\mathfrak so}}
\newcommand{\su}{\ensuremath{\mathfrak su}}
\newcommand{\esl}{\ensuremath{\mathfrak sl}}

\newcommand{\C}{\ensuremath{\mathbb C}}
\newcommand{\R}{\ensuremath{\mathbb R}}
\newcommand{\N}{\ensuremath{\mathbb N}}
\newcommand{\Z}{\ensuremath{\mathbb Z}}

\theoremstyle{definition}
\newtheorem{definition}{Definition}
\newtheorem{theorem}{Theorem}
\newtheorem{remark}{Remark}
\newtheorem{example}{Example}
\newtheorem{lemma}{Lemma}

\newcommand{\ds}{\displaystyle}

\begin{document}

\begin{center}
  {\bf  \Large Unitary $(\g,K)$ modules of $SU(2,1)$}
\end{center}
\vspace{5mm}

\begin{center}
  D. Kova\v{c}evi\'{c}{\footnote{e-mail:domagoj.kovacevic@fer.hr}}\\
  University of Zagreb, Faculty of Electrical Engineering and Computing,\\
  Unska 3, HR-10000 Zagreb, Croatia
\end{center}
\setcounter{page}{1}
\vspace{7mm}

\begin{abstract}
  Let $G=SU(2,1)$. In this paper we parametrize irreducible unitary
  $(\g,K)$ modules of $G$. The parametrization is done in two steps.
  Firstly, we parametrize irreducible $(\g,K)$ modules (Theorem \ref{tmgk}).
  In the second step we find unitary $(\g,K)$ modules (Theorem \ref{tmunit}).
  One can compare our results with \cite{hrk72} and \cite{hrk}.
\end{abstract}

\section{Introduction}

Let $G$ be a real reductive group. We will follow the definition of the real
reductive group from \cite{kn}. The main goal of the representation theory is
finding the unitary dual of the group $G$. One can approach to this problem
using $(\g,K)$ modules (see \cite{bal}) which correspond to admissible
representations. In the first step all irreducible $(\g,K)$ modules are
found. In the second step it remains to find unitary $(\g,K)$ modules.
The first step correspond to Langlands classification. The second problem
is still unsolved in general.

In this paper we find the unitary dual for the group $G=SU(2,1)$. The unitary
dual of $SU(n,1)$ is already found by Kraljevi\'{c} (see \cite{hrk72} and
\cite{hrk}). Why should anyone solve already solved problem? We hope that our
technique, which is applied for $SU(2,1)$, can be generalized and applied to
some other groups. Also, the construction of unitary $(\g,K)$ modules in this
case is very explicit.

It is important to emphasize that for each weight $\lambda\in\h^*$, given by
$\lambda=(n,m)$, $n\in\N$ and $m\in\Z$ (see Definition \ref{restok}), the
space of vectors spanned by $v_{nm}^1$ is one-dimensional. It is also valid
for $SU(r,1)$ for any $r\in\N$, but it is not valid for $SU(r,s)$ where
$r,s\in\N$ and $r,s>1$. We used that fact in our construction and it
simplified our calculations. Also, it is the reason why the unitary dual
of $SU(r,s)$ is not known in general case.

We start with the set of coefficients $\{a_{nm}b_{nm},c_{nm},d_{nm}\}$
which completely describe the action of the complexified Lie algebra \g.
Actually, products $ad$ and $bc$ can be calculated and they give an important
information about the irreducibility of $(\g,K)$ modules. Also, the set of
products $ad$ and $bc$ (explained in Theorem \ref{tmrel}) is the key ingredient
in description of unitary $(\g,K)$ modules and classification of irreducible
unitary $(\g,K)$ modules. The main idea in this approach is to treat $K$ types
as points. The rest of construction at some points looks like a construction of
irreducible unitary $(\g,K)$ modules of $SL(2,\R)$.

In Section \ref{secsl2r} we recall basic results in representation theory of
$SL(2,\R)$. Some statements will be used later and some statements will be
compared with our results. Also, we wanted to demonstrate our ideas in
this case. In Section \ref{secgk} we construct $(\g,K)$ modules using certain
set of coefficients. After that we parametrize irreducible $(\g,K)$ modules
of $G=SU(2,1)$. The parametrization is given in Theorem \ref{tmgk}.
In Section \ref{secunit} we parametrize unitary $(\g,K)$ modules.
The key step is done in Theorem \ref{tmun}.
The parametrization is given in Theorem \ref{tmunit}.

Lie groups will be denoted by capital letters, corresponding Lie algebras
by Gothic letters with subscript 0 and complexified Lie algebras by
Gothic letters without subscript. For example, the Lie algebra of $G$
will be denoted by \go\ and the complexified Lie algebra by \g. If $H$,
$X$ and $Y$ span a basis for \esl(2,\C) such that $[H,X]=2X$, $[H,Y]=-2Y$
and $[X,Y]=H$ then we say that \esl(2,\C) is represented by a
triple $(H,X,Y)$. If $\pi,V$ is the representation of \g, usually, we will
write $X.v$ instead of $\pi(X)v$ for $X\in\g$ and $v\in V$.

\section{Unitary dual of $SL(2,\R)$}\label{secsl2r}

Let $G=SL(2,\R),\ K=SO(2),\ \g=\esl(2,\C),\ \go=\esl(2,\R)$ and $\kk=\so(2)$.
Unitary representations of $G$ are well known, see \cite{bar}. However, we
redo the construction since some details appear later.

We use notation and results from \cite{vo}.
The basis of \g\ contains elements
\begin{equation}\nonumber
  H=-i\left[\begin{array}{cc}0&1\\-1&0\end{array}\right],
\end{equation}
\begin{equation}\nonumber
  X=\frac12\left(\left[\begin{array}{cc}1&0\\0&-1\end{array}\right]
    +i\left[\begin{array}{cc}0&1\\1&0\end{array}\right]\right)
    =\frac12\left(A+iB\right)
\end{equation}
and
\begin{equation}\nonumber
  Y=\frac12\left(\left[\begin{array}{cc}1&0\\0&-1\end{array}\right]
    -i\left[\begin{array}{cc}0&1\\1&0\end{array}\right]\right)
    =\frac12\left(A-iB\right).
\end{equation}
It is easy to check that
\begin{equation}\nonumber
  [iH,B]=2A,\quad [A,iH]=2B\quad\mbox{and}\quad[B,A]=-2iH.
\end{equation}
Let $W$ be a \esl(2,\C) module.
Then we can choose a basis $\{w^k\,|\,k\in S\subset\Z\}$ of $W$ such that
$w^k\in W$, $H.w^k=kw^k$ and
\begin{equation}\nonumber
  X.w^k=\frac12\left(\lambda+(k+1)\right)w^{k+2}=a_kw^{k+2}
\end{equation}
and
\begin{equation}\nonumber
  Y.w^k=\frac12\left(\lambda-(k-1)\right)w^{k-2}=b_kw^{k-2}
\end{equation}
for some $\lambda\in\C$ and some set $S$ (\cite{vo}, Lemma 1.2.6). In \cite{vo},
they analyze irreducible (\g,$K$) modules and do not mention the set $S$.
Here, at this point, we concentrate on coefficients $a_k$ and $b_k$.
The set $S$ can have the form $\{2m\,|\,m\in\Z\}$ or $\{1+2m\,|\,m\in\Z\}$.
If $\lambda\notin\Z$, then the module $W$ is irreducible. If $\lambda\in\Z$,
then the module $W$ has submodules. It is easy to see that
\begin{equation}\nonumber
  A.w^k=\left(X+Y\right).w^k=a_kw^{k+2}+b_kw^{k-2}
\end{equation}
and
\begin{equation}\nonumber
  B.w^k=i\left(-X+Y\right).w^k=i\left(-a_kw^{k+2}+b_kw^{k-2}\right).
\end{equation}
If we consider a finite-dimensional module $V$ of dimension $n$, we will use
the same basis, with different indexes denoted by $v$, such that
$H.v^k=(n+1-2k)v^k$ and
\begin{equation}\label{a5}
  X.v^k=-(k-1)v^{k-1}
\end{equation}
and
\begin{equation}\label{a7}
  Y.v^k=-(n-k)v^{k+1}.
\end{equation}
We can assume that $v^0=v^{n+1}=0$.

Let us determine $\lambda$s for which it is possible to construct an inner
product $\langle\cdot,\cdot\rangle:W\times W\rightarrow\C$
such that \eqref{c5} is satisfied.
Let us assume that $\langle w^k,w^l\rangle\neq0$ for some $k$ and $l$. Then,
\begin{equation}\nonumber
  \langle iH.w^k,w^l\rangle=\langle w^k,(iH)^*.w^l\rangle
\end{equation}
and \eqref{c5} show that
$ik\langle w^k,w^l\rangle=il\langle w^k,w^l\rangle$. Hence
\begin{equation}\label{a15}
  \langle w^k,w^l\rangle=0\quad\mbox{for all}\ k\neq l.
\end{equation}
Now, from \eqref{a15}, it follows that
\begin{equation}\nonumber
  A^*.w^k=\overline{b_{k+2}}\frac{||w^k||^2}{||w^{k+2}||^2}w^{k+2}+
    \overline{a_{k-2}}\frac{||w^k||^2}{||w^{k-2}||^2}w^{k-2}
\end{equation}
and
\begin{equation}\nonumber
  B^*.w^k=i\left(-\overline{b_{k+2}}\frac{||w^k||^2}{||w^{k+2}||^2}w^{k+2}+
    \overline{a_{k-2}}\frac{||w^k||^2}{||w^{k-2}||^2}w^{k-2}\right).
\end{equation}
Each time we get the same condition:
\begin{equation}\nonumber
  a_k+\overline{b_{k+2}}\frac{||w^k||^2}{||w^{k+2}||^2}=0,\quad\forall k.
\end{equation}
or
\begin{equation}\nonumber
  \frac{a_k}{\overline{b_{k+2}}}=-\frac{||w^k||^2}{||w^{k+2}||^2}
    \in\left(-\infty,0\right),\quad\forall k.
\end{equation}
However, we prefer to multiply it by $b_{k+2}\overline{b_{k+2}}=
||b_{k+2}||^2$ and say that the irreducible representation
is unitary if and only if
\begin{equation}\label{a20}
  a_kb_{k+2}\in(-\infty,0),\quad\forall k.
\end{equation}
We consider an open interval $(-\infty,0)$ since the case $a_kb_{k+2}=0$ leads
to reducibility. It will be explained in Remark \ref{rmsub}.
Relation \eqref{a20} transforms to
\begin{equation}\nonumber
  (\lambda+k+1)(\lambda-(k+1))=\lambda^2-(k+1)^2\in\left(-\infty,0\right)
    ,\quad\forall k.
\end{equation}
Let us recall that $k=2z$ or $k=2z+1$ for $z\in\Z$.
If $k=2z$, $\lambda$ can be equal to $ri$ for $r\in\R$ (it corresponds to
principal series), $r\in(-1,1)$ (it corresponds to complementary series) and
$2m+1$ for $m\in\Z$ (submodules correspond to discrete seris and the trivial
representation). If $k=2z+1$, $\lambda$ can be equal to $ri$ for $r\in\R^*$
(it corresponds to prinicipal series), $0$ (submodules correspond to mock
discrete series) and $2m$ for $m\in\Z$ (submodules correspond to discrete
series).

It remains to analyze (unitary) finite-dimensional representations of the
compact real form of \g. For the beginning, let us choose the basis:
\begin{equation}\nonumber
  -\frac12iH,\quad\frac12(X-Y)=\frac12iB\quad\mbox{and}\quad
    -\frac12i(X+Y)=-\frac12iA.
\end{equation}
Then $\ds -\frac12iH.v^k=-\frac12i(n+1-2k)v^k$,
\begin{equation}\nonumber
  \frac12(X-Y).v^k=\frac12\left(-(k-1)v^{k-1}+(n-k)v^{k+1}\right)
\end{equation}
and
\begin{equation}\nonumber
  -\frac12i(X+Y).v^k=\frac12i\left((k-1)v^{k-1}+(n-k)v^{k+1}\right).
\end{equation}
Element $\ds -\frac12iH$ satisfies \eqref{c5} (for any inner product satisfying
\eqref{a15}). It remains to analyze remaining two elements. In any case,
\eqref{c5} produces the same condition:
\begin{equation}\label{a23}
  ||v^{k+1}||^2=\frac{k}{n-k}||v^k||^2.
\end{equation}
It will be useful to express $||v^k||^2$ in terms of $||v^1||^2$.
Using induction, one can easily show that
\begin{equation}\label{a25}
  ||v^k||^2=\frac{(k-1)!(n-k)!}{(n-1)!}||v^1||^2=
  \frac{1}{\binom{n-1}{k-1}}||v^1||^2
\end{equation}
\begin{example}
Let $V$ be the $\su(2)$ module such that $\dim V=5$ and $||v^1||=1$. Then
\begin{equation}\nonumber
  ||v^2||=\frac{1}{2},\quad||v^3||=\frac{1}{\sqrt{6}},\quad
  	||v^4||=\frac{1}{2},\quad||v^5||=1.
\end{equation}
\end{example}

\section{$(\g,K)$ modules for $SU(2,1)$}\label{secgk}

Let $G=SU(2,1)$. Then $K=S(U(2)\times U(1))=SU(2)\times S^1$,
$\kko=\su(2)\oplus\R$ and $\kk=\esl(2,\C)\oplus\C$.
This \esl(2,\C) can be represented by a triple $(H_\alpha,X_\alpha,Y_\alpha)$.
It remains to set $\C=Z\C$ where $Z=H_\alpha+2H_\beta$. Let us define a basis
for \g. The Cartan subalgebra $\h\subset\kk$ is generated by
\begin{equation}\nonumber
  H_{\alpha}=\left[\begin{array}{ccc}
    1&0&0\\0&-1&0\\0&0&0
    \end{array}\right]\quad\mbox{and}\quad\quad
  H_{\beta}=\left[\begin{array}{ccc}
    0&0&0\\0&1&0\\0&0&-1
    \end{array}\right].
\end{equation}
Now, we define
\begin{equation}\nonumber
  X_{\alpha}=\left[\begin{array}{ccc}
    0&1&0\\0&0&0\\0&0&0
    \end{array}\right],\quad
  X_{\beta}=\left[\begin{array}{ccc}
    0&0&0\\0&0&1\\0&0&0
    \end{array}\right]\quad\mbox{and}\quad
  X_{\alpha+\beta}=\left[\begin{array}{ccc}
    0&0&1\\0&0&0\\0&0&0
    \end{array}\right].
\end{equation}
Elements $Y_{\alpha}$, $Y_{\beta}$ and $Y_{\alpha+\beta}$ are defined similarly.
The next step is to define a basis for \go. We take $iH_{\alpha}$
and $iH_{\beta}$ for the Cartan subalgebra and
\begin{equation}\nonumber
  A_{\alpha}=X_{\alpha}-Y_{\alpha}\quad\mbox{and}\quad
  B_{\alpha}=i(X_{\alpha}+Y_{\alpha})
\end{equation}
for the remainder of \kko. The rest of \go\ is given by
\begin{equation}\nonumber
A_{\beta}=X_{\beta}+Y_{\beta},\quad\quad
B_{\beta}=i(X_{\beta}-Y_{\beta})
\end{equation}
\begin{equation}\label{b5}
A_{\alpha+\beta}=X_{\alpha+\beta}+Y_{\alpha+\beta}\quad\mbox{and}\quad
B_{\alpha+\beta}=i(X_{\alpha+\beta}-Y_{\alpha+\beta})
\end{equation}
One should notice a different sign in expressions for $\alpha$ and $\beta$.

Let us consider irreducible representations of $K$. Since
$\kk=\esl(2,\C)\oplus\C Z$, where $Z=H_\alpha+2H_\beta$, irreducible
representations of \kk\ are irreducible representations of \esl(2,\C) on
which $Z$ acts as a multiplication by scalars. Since \C Z ($\R iZ$)
corresponds to a circle ($K$ is compact), the scalar $m$ is an integer.

\begin{definition}\label{restok}
  We denote $K$ modules by $V_{nm}$ where $n$ is the dimension of the space
  $V_{nm}$ and $m$ is the scalar by which $Z=H_\alpha+2H_\beta$ acts on that
  space. Let $\{v_{nm}^k\}$, $k\in\{1,2,\ldots,n\}$ be a basis for $V_{nm}$
  defined above ($H_\alpha.v^k=(n+1-2k)v^k$, $X.v^k=-(k-1)v^{k-1}$ (as in
  \eqref{a5}) and $Y.v^k=-(n-k)v^{k+1}$).
\end{definition}

Hence, $Z.v_{nm}^k=(H_\alpha+2H_\beta).v_{nm}^k=mv_{nm}^k$
for any element $v_{nm}^k\in V_{nm}$.
Let $V$ be an irreducible $(\g,K)$ module. Then the restriction of $V$ to $K$
has the form
\begin{equation}\nonumber
    V|_K=\bigoplus_{n\in\N,\,m\in\Z}V_{nm}.
\end{equation}
It should be more correct to write $(n,m)\in S(V)$ where
$S(V)\subset \N\times\Z$ is some set which depends on $V$.
For example, it is easy to see that $n+m$ is an odd number. However, we want
to emphasize that $n$ in the natural number and $m$ is an integer.
This set $S(V)$ will be discussed later (after the following theorem which
describes relationship among $K$ types). Now, it is important to emphasize
that we know that the multiplicity of $K$ modules $V_{nm}$ is 1 or,
equivalently, the dimension of the space spanned by the vectors $v_{nm}^1$
is 1 for all $n\in\N$ and $m\in\Z$. It makes this construction possible.
If our group $G$ is more complicated ($SU(2,2)$), multiplicities are bigger
then 1 and the construction is more complex. The author currently works on
this problem.

\begin{theorem}\label{tmdek}
  Let $V$ be an irreducible $(\g,K)$ module and let $\{v_{nm}^k\}$ be a basis
  for $V_{nm}$ as in Definition \ref{restok}. Then
  \begin{align}
    &X_{\alpha+\beta}.v_{nm}^k=a_{nm}v_{n+1\,m+3}^k+
      \frac{k-1}{n-1}c_{nm}v_{n-1\,m+3}^{k-1},\label{b7}\\
    &X_\beta.v_{nm}^k=-a_{nm}v_{n+1\,m+3}^{k+1}+
      \frac{n-k}{n-1}c_{nm}v_{n-1\,m+3}^k,\nonumber\\
    &Y_{\alpha+\beta}.v_{nm}^k=b_{nm}v_{n+1\,m-3}^{k+1}+
    \frac{n-k}{n-1}d_{nm}v_{n-1\,m-3}^k,\label{b9}\\
    &Y_\beta.v_{nm}^k=b_{nm}v_{n+1\,m-3}^k-
    \frac{k-1}{n-1}d_{nm}v_{n-1\,m-3}^{k-1}.\nonumber
  \end{align}
  for some coefficients $a_{nm}$, $b_{nm}$, $c_{nm}$ and $d_{nm}$.
\end{theorem}

\begin{remark}\label{ktyp}
	Possible values of $n$ and $m$ in $V$ will be determined gradually later.
	At this moment we can say that the set $\{v_{nm}^1\}$ looks like a cone
	(or a subset of a cone). Also, one should notice that if $V_{nm}$ is a
	$K$ type of some $(\g,K)$ module $V$, then possible $K$ types  of $V$for
	fixed $n$ and have the form $V_{n\,m+6k}$ for $k\in\Z$ and $K$ types of
	the form $V_{n\,m+6k+2}$ and $V_{n\,m+6k+4}$ belong to some other $(\g,K)$
	modules which have no common $K$ types with $V$. Also, $V_{n\,m+2k+1}$,
	$k\in\Z$, are never $K$ types for any $(\g,K)$ module.
\end{remark}

\begin{remark}
	One can ask if the value of $n$ can be equal to 1 in \eqref{b7} and
	\eqref{b9} since the denominator of some fractions is $n-1$. Let us
	analyze the statement of the theorem. The action of $X_{\alpha+\beta}$
	and $X_\beta$ on vectors in $V_{nm}$ produces vectors which are sums
	of vectors from $V_{n+1\,m+3}$ and $V_{n-1\,m+3}$. It is one of main ideas
	of this paper: the action is not "wild", it can be understood completely.
	If $n=1$, the space $V_{0\,m-3}$ does not exist (the dimension of that
	space is 0) and it should be more
	correct to write $X_{\alpha+\beta}.v_{1m}^k=a_{nm}v_{2\,m+3}^k$ and
	$X_\beta.v_{1m}^k=-a_{nm}v_{2\,m+3}^{k+1}$ instead of \eqref{b7}.
	However, we did not want to write it as separate statements.
	Similar statements are valid for $Y_{\alpha+\beta}$ and $Y_\beta$.
	Finally, $c_{1m}$ and $d_{1m}$ will be 0 in our calculations.
\end{remark}

\begin{remark}
  This theorem shows that the set $\{V_{nm}\}$ and coefficients $a_{nm}$,
  $b_{nm}$, $c_{nm}$ and $d_{nm}$ determine the structure of the $(\g,K)$
  modules $V$. It is clear that these coefficients are not determined uniquely.
  However, the products $a_{nm}d_{n+1\,m+3}$ and $b_{nm}c_{n+1\,m-3}$ are
  unique and it is an important observation.
\end{remark}
  
\begin{remark}
  One can say that irreducible $K$ modules $V_{nm}$ can be represented as
  points. It is the main idea in our construction. We want to understand the
  structure of $K$ modules.
\end{remark}

\begin{proof}
We will prove the first two relations.
Let us consider $X_{\alpha+\beta}.v_{nm}^1$. It is clear that
\begin{equation}\nonumber
  X_{\alpha+\beta}.v_{nm}^1\in\bigoplus_{p\geq n+1}V_{p\,m+3}.
\end{equation}
Let us assume that $X_{\alpha+\beta}.v_{nm}^1=a+b$ for $a\in V_{q\,m+3}$,
where $q>n+1$ and $b\in\bigoplus_{p\geq n+1,\,p\neq q}V_{p\,m+3}$.
Then $X_\alpha X_{\alpha+\beta}.v_{nm}^1\neq0$. It produces a contradiction
since $X_{\alpha+\beta}X_\alpha.v_{nm}^1=0$ and $[X_\alpha,X_{\alpha+\beta}]=0$.
It shows that
\begin{equation}\label{b10}
  X_{\alpha+\beta}.v_{nm}^1=a_{nm}v_{n+1\,m+3}^1
\end{equation}
for some coefficient $a_{nm}$. Now, let us calculate $X_\beta v_{nm}^1$.
Similar calculation shows that
\begin{equation}\nonumber
X_\beta.v_{nm}^1=\lambda v_{n+1\,m+3}^2+c_{nm}v_{n-1\,m+3}^1.
\end{equation}
for some coefficients $\lambda$ and $c_{nm}$. The action of $X_\alpha$ and
\eqref{a5} produces
\begin{equation}\nonumber
X_\alpha X_\beta.v_{nm}^1=-\lambda v_{n+1\,m+3}^1.
\end{equation}
Since $[X_\alpha,X_\beta]=X_{\alpha+\beta}$ and $X_\alpha.v_{nm}^1=0$,
\begin{equation}\label{b15}
  X_{\alpha+\beta}.v_{nm}^1=[X_\alpha,X_\beta].v_{nm}^1=-\lambda v_{n+1\,m+3}^1.
\end{equation}
Now, \eqref{b10} and \eqref{b15} show that $\lambda=-a_{nm}$.

We continue by induction on $k$. The base of induction is just proved.
Let us assume that first two relations are valid for $k$. Since
$[Y_\alpha,X_\beta]=0$,
\begin{align}\nonumber
  X_\beta.v_{nm}^{k+1}&=-\frac1{n-k}X_\beta Y_\alpha.v_{nm}^k
    =-\frac1{n-k}Y_\alpha X_\beta.v_{nm}^k\nonumber\\
  &=-\frac1{n-k}Y_\alpha.\left(-a_{nm}v_{n+1\,m+3}^{k+1}+
    \frac{n-k}{n-1}c_{nm}v_{n-1\,m+3}^k\right)\nonumber\\
  &=\frac{a_{nm}}{n-k}(-(n-k))v_{n+1\,m+3}^{k+2}-
    \frac{c_{nm}}{n-1}(-(n-1-k))v_{n-1\,m+3}^{k+1}\nonumber\\
  &=-a_{nm}v_{n+1\,m+3}^{k+2}+
    \frac{n-(k+1)}{n-1}c_{nm}v_{n-1\,m+3}^{k+1}.\nonumber
\end{align}
Now we use this result and obtain
\begin{gather}
  X_{\alpha+\beta}.v_{nm}^{k+1}=\left(X_\alpha X_\beta-X_\beta X_\alpha\right).
    v_{nm}^{k+1}\nonumber\\
  =X_\alpha.\left(-a_{nm}v_{n+1\,m+3}^{k+2}+
    \frac{n-(k+1)}{n-1}c_{nm}v_{n-1\,m+3}^{k+1}\right)-
    X_\beta.\left((-k)v_{nm}^k\right)\nonumber\\
  =a_{nm}(k+1)v_{n+1\,m+3}^{k+1}-
    \frac{n-(k+1)}{n-1}kc_{nm}v_{n-1\,m+3}^k\nonumber\\
  +k\left(-a_{nm}v_{n+1\,m+3}^{k+1}+
    \frac{n-k}{n-1}c_{nm}v_{n-1\,m+3}^k\right)\nonumber\\
  =a_{nm}v_{n+1\,m+3}^{k+1}+k\left(-\frac{n-(k+1)}{n-1}
    +\frac{n-k}{n-1}\right)c_{nm}v_{n-1\,m+3}^k\nonumber\\
  =a_{nm}v_{n+1\,m+3}^{k+1}+\frac k{n-1}c_{nm}v_{n-1\,m+3}^k.\nonumber
\end{gather}
Remaining two relations can be proved similarly.
\end{proof}

\begin{theorem}\label{tmrel}
	Let $V$ be an irreducible $(\g,K)$ module. Then coefficients
	$a_{nm}$, $b_{nm}$, $c_{nm}$ and $d_{nm}$ satisfy following relations,
	\begin{align}
	  &-\frac1na_{nm}d_{n+1\,m+3}+b_{nm}c_{n+1\,m-3}-c_{nm}b_{n-1\,m+3}=
	   \frac{m-n+1}2,\label{b20}\\
	  &-a_{nm}d_{n+1\,m+3}+\frac1nb_{nm}c_{n+1\,m-3}+d_{nm}a_{n-1\,m-3}=
	  \frac{m+n-1}2,\label{b25}\\
	  &a_{nm}d_{n+1\,m+3}=d_{nm}a_{n-1\,m-3}=b_{nm}c_{n+1\,m-3}=
	   c_{nm}b_{n-1\,m-3}=0\nonumber\\
	  &\hspace*{10mm}\mbox{when }V_{nm}\mbox{ is }K\;\mbox{type and }
	   V_{n\pm1\,m\pm3}\mbox{ is not }K\;\mbox{type}\label{b27}\\
	  &b_{nm}a_{n+1\,m-3}=a_{nm}b_{n+1\,m+3},\label{b30}\\
	  &d_{nm}c_{n-1\,m-3}=c_{nm}d_{n-1\,m+3},\label{b35}\\
	  &(n+1)a_{nm}c_{n+1\,m+3}=nc_{nm}a_{n-1\,m+3},\label{b40}\\
	  &(n+1)b_{nm}d_{n+1\,m-3}=nd_{nm}b_{n-1\,m-3}\label{b45}.
	\end{align}
	for all $n\in\N$ and $m\in\Z$ for which $V_{nm}$ is $K$ type of $V$.
	If the set of $K$ types is given together with relations
	\eqref{b20} -- \eqref{b45}, then it is possible to reconstruct an
	irreducible $(\g,K)$ module $V$.
\end{theorem}

\begin{remark}
	The theorem enumerates necessary and sufficient conditions for the
	existence of (irreducible) $(\g,K)$ module $V$. 
\end{remark}

\begin{proof}
Let us assume that $(\g,K)$ module exists.
Let us consider	a nonzero element $v_{nm}^k$ and apply the relation
$[X_\beta,Y_\beta]=H_\beta$ on that element. The left hand side is equal to
\begin{gather}
  (X_\beta Y_\beta-Y_\beta X_\beta).v_{nm}^k\nonumber\\
  =X_\beta.\left(b_{nm}v_{n+1\,m-3}^k-
    \frac{k-1}{n-1}d_{nm}v_{n-1\,m-3}^{k-1}\right)\nonumber\\
  -Y_\beta.\left(-a_{nm}v_{n+1\,m+3}^{k+1}+
    \frac{n-k}{n-1}c_{nm}v_{n-1\,m+3}^k\right)\nonumber\\
  =b_{nm}\left(-a_{n+1\,m-3}v_{n+2\,m}^{k+1}+
    \frac{n+1-k}nc_{n+1\,m-3}v_{nm}^k\right)\nonumber\\
  -\frac{k-1}{n-1}d_{nm}\left(-a_{n-1\,m-3}v_{nm}^k+
    \frac{n-k}{n-2}c_{n-1\,m-3}v_{n-2\,m}^{k-1}\right)\nonumber\\
  +a_{nm}\left(b_{n+1\,m+3}v_{n+2\,m}^{k+1}-
    \frac knd_{n+1\,m+3}v_{nm}^k\right)\nonumber\\
  -\frac{n-k}{n-1}c_{nm}\left(b_{n-1\,m+3}v_{nm}^k-
    \frac{k-1}{n-2}d_{n-1\,m+3}v_{n-2\,m}^{k-1}\right).\nonumber
\end{gather}
The right hand side is equal to
\begin{equation}\nonumber
  H_\beta.v_{nm}^k=\frac{m-n-1+2k}2v_{nm}^k.
\end{equation}
One can compare coefficients of $v_{n-2\,m}^{k-1}$ and obtain \eqref{b35}.
Similarly, coefficient of $v_{n+2\,m}^{k+1}$ produces \eqref{b30}.
Finally, coefficient of $v_{nm}^k$ produces
\begin{gather}
  -\frac kna_{nm}d_{n+1\;m+3}+
    \frac{n+1-k}nb_{nm}c_{n+1\,m-3}\nonumber\\
  -\frac{n-k}{n-1}c_{nm}b_{n-1\,m+3}+
    \frac{k-1}{n-1}d_{nm}a_{n-1\,m-3}=
    \frac{m-n-1+2k}2.\label{b50}
\end{gather}
Also, $a_{nm}d_{n+1\;m+3}=0$ if $V_{n+1\,m+3}$ is not $K$ type of $V$,
$b_{nm}c_{n+1\;m-3}=0$ if $V_{n+1\,m-3}$ is not $K$ type of $V$,
$c_{nm}b_{n-1\;m+3}=0$ if $V_{n-1\,m+3}$ is not $K$ type of $V$ and
$d_{nm}a_{n-1\;m-3}=0$ if $V_{n-1\,m-3}$ is not $K$ type of $V$ and it is
\eqref{b27}.
For $k=1$, one obtains \eqref{b20} and for $k=n$ it transforms to \eqref{b25}.
It is easy to check that \eqref{b50} is a linear combination of \eqref{b20}
and \eqref{b25}, namely
\begin{equation}\nonumber
  \frac{n-k}{n-1}\eqref{b20}+\frac{k-1}{n-1}\eqref{b25}=\eqref{b50}.
\end{equation}
It shows that it is enough to consider \eqref{b20} and \eqref{b25}. These
two relations are more convenient then \eqref{b50} since $k$ does not appear
in \eqref{b20} and \eqref{b25}.

If we apply the relation $X_\beta X_{\alpha+\beta}=X_{\alpha+\beta}X_\beta$
on the element $v_{nm}^k$ we obtain \eqref{b40}.
Finally, if we apply the relation
$Y_\beta Y_{\alpha+\beta}=Y_{\alpha+\beta}Y_\beta$ on the element $v_{nm}^k$
we obtain \eqref{b45}.
One can check all other commutation relations in \g, but it will not produce
new conditions on coefficients $a_{nm}$, $b_{nm}$, $c_{nm}$ and $d_{nm}$. It
shows that \eqref{b20} -- \eqref{b45} have to be satisfied.

Now, let us assume that the set of $K$ types together with
coefficients $a_{nm}$, $b_{nm}$, $c_{nm}$ and $d_{nm}$ are given and
\eqref{b20} -- \eqref{b45} are satisfied.
At the beginning of the proof of the Theorem \ref{tmgk}, we will show that
$K$ types (or vectors $v_{nm}^1$) form a cone or a subset of a cone (a strip
or a parallelogram). The proof is technical and follows directly from
\eqref{b20} -- \eqref{b45}. Also, $a_{nm},b_{nm},c_{nm}$ and $d_{mm}$ are
different from 0 if $V_{nm}$ and $V_{n\pm1\,m\pm3}$ are $K$ types of an
irreducible $(\g,K)$ module $V$. Let us choose any $K$ type $V_{nm}$ of $V$
and $v_{nm}^1\in V_{nm}$. Then vectors $v_{nm}^k$ are defined by \eqref{a7},
where $Y$ in \eqref{a7} is $Y_\alpha$. Vectors $v_{n+1\,m\pm3}^1$ are defined
by $v_{n+1\,m+3}^1=\frac1{a_{nm}}X_{\alpha+\beta}.v_{nm}^1$ (by \eqref{b7}) and
$v_{n+1\,m-3}^1=\frac1{b_{nm}}Y_\beta.v_{nm}^1$. Vectors $v_{n+1\,m\pm3}^k$ are
defined by \eqref{a7} (again $Y$ in \eqref{a7} is $Y_\alpha$). Vectors
$v_{n-1\,m\pm3}^1$ are defined by
$v_{n-1\,m+3}^1=\frac1{c_{nm}}(X_\beta.v_{nm}^1+a_{nm}v_{n+1\,m+3}^2)$ and
$v_{n-1\,m-3}^1=\frac1{d_{nm}}(Y_{\alpha+\beta}.v_{nm}^1-b_{nm}v_{n+1\,m-3}^2)$.
Again, vectors $v_{n-1\,m\pm3}^k$ are defined by \eqref{a7}. Since, the
structure of $K$ types is simple, all $K$ types can be reached in this way.
One can ask if this
definition is good, or equivalently, is it possible to get two different
values for the same vector. It is not possible since all commutation relations
are satisfied. For example $v_{n+2\,m}^1$ can be obtained as
\begin{equation}\nonumber
    v_{n+2\,m}^1=\frac1{a_{n+1\,m-3}}X_{\alpha+\beta}.
    \frac1{b_{nm}}Y_\beta.v_{nm}^1
\end{equation}
and
\begin{equation}\nonumber
    v_{n+2\,m}^1=\frac1{b_{n+1\,m+3}}Y_\beta.
    \frac1{a_{nm}}X_{\alpha+\beta}.v_{nm}^1
\end{equation}
The commutator of $X_{\alpha+\beta}$ and $Y_\beta$ is $X_\alpha$, but
$X_\alpha.v_{nm}^1=0$. Also $a_{n+1\,m-3}b_{nm}=b_{n+1\,m+3}a_{nm}$ by
\eqref{b30}. We can also check it for $k>1$. If we start with the
definition of $v_{n+1\,m+3}^1$,
\begin{equation}\nonumber
    v_{n+1\,m+3}^1=\frac1{a_{nm}}X_{\alpha+\beta}.v_{nm}^1
\end{equation}
and act by $Y_\alpha$ on it, we obtain
\begin{equation}\nonumber
    Y_\alpha.v_{n+1\,m+3}^1=\frac1{a_{nm}}\left(X_{\alpha+\beta}Y_\alpha+
    X_\beta\right).v_{nm}^1.
\end{equation}
One can apply formulas from Theorem \ref{tmdek} and get equality. However,
using these formulas will repeat the proof and give the spirit of that theorem.
\end{proof}

It is more convenient to work with
$a_{nm}d_{n+1\,m+3}$ and $b_{nm}c_{n+1\,m-3}$ then $a_{nm}$, $b_{nm}$, $c_{nm}$
and $d_{nm}$. The reason is very simple. The later expressions are not
determined uniquely since they depend on the choice of vectors $v_{nm}^k$.
Hence, we plan to determine expressions $a_{nm}d_{n+1\,m+3}$ and
$b_{nm}c_{n+1\,m-3}$ using \eqref{b20} and \eqref{b25} and
then show that it is possible to determine coefficients
$a_{nm}$, $b_{nm}$, $c_{nm}$ and $d_{nm}$ such that all relations above are
satisfied (and given $(\g,K)$ module exists). We will be able to give explicit
formulas for expressions $a_{nm}d_{n+1\,m+3}$ and $b_{nm}c_{n+1\,m-3}$.
Then, it is easy to give formulas for coefficients
$a_{nm}$, $b_{nm}$, $c_{nm}$ and $d_{nm}$.

Let us write \eqref{b45} for $m+6$ instead of $m$ and multiply by \eqref{b40}.
It produces
\begin{gather}
  (n+1)^2a_{nm}d_{n+1\,m+3}b_{n\,m+6}c_{n+1\,m+3}\nonumber\\
  =n^2a_{n-1\,m+3}d_{n\,m+6}b_{n-1\,m+3}c_{nm}\label{b55}
\end{gather}
and
\begin{equation}\label{b60}
  \frac{a_{nm}d_{n+1\,m+3}}{a_{n-1\,m+3}d_{n\,m+6}}
  \frac{b_{n\,m+6}c_{n+1\,m+3}}{b_{n-1\,m+3}c_{nm}}=\frac{n^2}{(n+1)^2}.
\end{equation}
Now, \eqref{b30} and \eqref{b35} show that
$\ds \frac{a_{nm}}{a_{n-1\,m+3}}=\frac{b_{n\,m+6}}{b_{n-1\,m+3}}$ and
$\ds \frac{d_{n+1\,m+3}}{d_{n\,m+6}}=\frac{c_{n+1\,m+3}}{c_{nm}}$.
We conclude that \eqref{b60} transforms to
\begin{equation}\nonumber
  \left(\frac{a_{nm}d_{n+1\,m+3}}{a_{n-1\,m+3}d_{n\,m+6}}\right)^2
  =\frac{n^2}{(n+1)^2}
\end{equation}
or
\begin{equation}\nonumber
  \frac{a_{nm}d_{n+1\,m+3}}{a_{n-1\,m+3}d_{n\,m+6}}=
  \frac{b_{n\,m+6}c_{n+1\,m+3}}{b_{n-1\,m+3}c_{nm}}=\pm\frac{n}{n+1}.
\end{equation}
It is possible to give more precise statement. Using induction, one can obtain
\begin{equation}\nonumber
  \frac{a_{nm}d_{n+1\,m+3}}{a_{n-1\,m+3}d_{n\,m+6}}=
  \frac{b_{n\,m+6}c_{n+1\,m+3}}{b_{n-1\,m+3}c_{nm}}=\frac{n}{n+1}.
\end{equation}
This relation will be a consequence of formulas \eqref{b75} and \eqref{b80}.
However, we mention it now in order to give a better insight into the
structure of coefficients $a_{nm}$, $b_{nm}$, $c_{nm}$ and $d_{nm}$.

\begin{theorem}\label{tmgk}
  For any $c\in\C$ and $t\in\Z$ there exist a $(\g,K)$ module $V(c,2t)$
  such that
  \begin{equation}\nonumber
    V(c,2t)|_K=V_{1\,2t}\oplus\bigoplus_{n,m\in\Z,\,n>1}V_{nm}
  \end{equation}
  and $a_{1\,2t}d_{2\,2t+3}=c-\frac12t$. This module can be reducible. Any other
  irreducible $(\g,K)$ module $V$ is a submodule, quotient or subquotient
  of some $V(c,2t)$.
\end{theorem}

\begin{remark}\label{rmsub}
	We are concentrated on irreducible $(\g,K)$ modules. Reducibility of
	modules $V(c,2t)$ will be obtained when some product(s) $a_{nm}d_{n+1\,m+3}$
	or $b_{nm}c_{n+1\,m-3}$ are equal to 0. The theorem says that
	irreducible $(\g,K)$ modules can be obtained a submodules, quotients
	or subquotients. Once we have formulas \eqref{b85} -- \eqref{b100},
	it will be possible to determine if we have a submodule, quotient or
	subquotient. However, it will not be important for us. We are looking for
	irreducible $(\g,K)$ modules and their description. Finally, it is,
	maybe, possible to find another choice of coefficients
	$a_{nm}$, $b_{nm}$, $c_{nm}$ and $d_{nm}$ in \eqref{b85} -- \eqref{b100}
	and it would lead to another relationship among our modules.
	Hence, by abuse of notation, we will say just a submodule.
\end{remark}

\begin{proof}
Let us put $a_{1\,2t}d_{2\,2t+3}=c-\frac12t$.
Then \eqref{b20} (and also \eqref{b25}) shows that
\begin{equation}\nonumber
  b_{1\,2t}c_{2\,2t-3}=c+\frac12t.
\end{equation}
If $a_{1+k\,2t+3k}d_{2+k\,2t+3+3k}\neq0$,
then $b_{1+k\,2t+6+3k}c_{2+k\,2t+3+3k}=0$ for $k\geq0$ (by \eqref{b55}).
If $b_{1+k\,2t-3k}c_{2+k\,2t-3-3k}\neq0$,
then $a_{1+k\,2t-6-3k}a_{2+k\,2t-3-3k}=0$ for $k\geq0$.
It means that $K$ modules of the irreducible component of $V(c,2t)$ can form
a cone (if $a_{1+k\,2t+3k}d_{2+k\,2t+3+3k}\neq0$ and
$b_{1+k\,2t-3k}c_{2+k\,2t-3-3k}\neq0$ for $k\geq0$),
a strip (if $a_{1+k\,2t+3k}d_{2+k\,2t+3+3k}=0$ for some $k\in\N\cup\{0\}$ or
$b_{1+k\,2t-3k}c_{2+k\,2t-3-3k}=0$ for some $k\in\N\cup\{0\}$) or a
parallelogram (if $a_{1+k\,2t+3k}d_{2+k\,2t+3+3k}=0$ for some $k\in\N\cup\{0\}$
and $b_{1+l\,2t-3l}c_{2+l\,2t-3-3l}=0$ for some $l\in\N\cup\{0\}$).

We claim that our expressions $ad$ and $bc$ are determined uniquely. Relations
\eqref{b20} and \eqref{b25}, for $n>1$, produce two independent equations.
It is enough to walk from one vertex to another where two expressions
$ad$ and $bc$ are already determined and calculate remaining two.
A reader can easily reconstruct the path.
Formulas for the vertex which is obtained by moving $p$ steps in the
$\alpha+\beta$ direction and $q$ steps in $-\beta$ direction, have the form
\begin{equation}\label{b75}
  a_{1+p+q\,2t+3p-3q}d_{2+p+q\,2t+3p-3q+3}=\frac{p+1}{p+q+2}
    \left(2c-(p+1)t-p(p+2)\right)
\end{equation}
and
\begin{equation}\label{b80}
  b_{1+p+q\,2t+3p-3q}c_{2+p+q\,2t+3p-3q-3}=\frac{q+1}{p+q+2}
    \left(2c+(q+1)t-q(q+2)\right).
\end{equation}
They can be checked directly. One could write $n+1$ instead of $p+q+2$ in
denominators of \eqref{b75} and \eqref{b80}.

Coefficients $a_{nm}$, $b_{nm}$, $c_{nm}$ and $d_{nm}$ can be defined by
\begin{align}
  &a_{1+p+q\,2t+3p-3q}=2c-(p+1)t-p(p+2),\label{b85}\\
  &b_{1+p+q\,2t+3p-3q}=2c+(q+1)t-q(q+2),\nonumber\\
  &c_{2+p+q\,2t+3p-3q-3}=\frac{q+1}{p+q+2},\nonumber\\
  &d_{2+p+q\,2t+3p-3q+3}=\frac{p+1}{p+q+2}.\label{b100}
\end{align}
One can check that \eqref{b30} -- \eqref{b45} are satisfied.
By Theorem \ref{tmrel}, $(\g,K)$ module $V(c,2t)$ is well defined.

Now, it remains to show that any other module is a submodule (see
Remark \ref{rmsub}) of some $V(c,2t)$.
Let us consider some irreducible $(\g,K)$ module $W(r,s)$, $r\in\N$, $r>1$ and
$s\in\Z$ of the form
\begin{equation}\nonumber
  W(r,s)|_K=V_{rs}\oplus\bigoplus_{n,m\in\Z,\,n>r}V_{nm}
\end{equation}
Let us notice that $r+s=2z+1$ for some $z\in\Z$.
This time (since $r>1$), the system of two equations produced by
\eqref{b20} and \eqref{b25} has a unique solution. It shows that coefficients
$ad$ and $bc$ are uniquely determined for the fixed choice of $r$ and $s$.
It remains to show that $W(r,s)$ is a submodule of some $V(c,2t)$ but it is
straightforward: $W(r,s)$ is submodule of
\begin{equation}\nonumber
  V\left(\frac{(r-1)(-r+1+s)-2}{4},-3r+3+s\right)
\end{equation}
and
\begin{equation}\nonumber
  V\left(\frac{(r-1)(-r+1-s)-2}{4},3r-3+s\right).
\end{equation}
Since $r+s=2z+1$, $-3r+3+s=2z-4r+4$.
\end{proof}

\begin{example}
Let us consider the module $W(4,3)$. It is a submodule of
$\ds V\left(-\frac12,-6\right)$ and $V(-5,12)$. It is a nice exercise to
calculate products $a_{nm}d_{n+1\,m+3}$ and $b_{nm}c_{n+1\,m-3}$ for $W(4,3)$
using \eqref{b20} and \eqref{b25} (solving system for each vertex) and
compare with \eqref{b75} and \eqref{b80} for $\ds V\left(-\frac12,-6\right)$
and $V(-5,12)$.
\end{example}

\section{Unitary dual of $SU(2,1)$}\label{secunit}

For the beginning we give a definition of unitary $(\g,K)$ modules for any
real reductive group.

\begin{definition}\label{defunit}
	Unitary $(\g,K)$ module $V$ is a $(\g,K)$ module equipped with the inner
	product $\langle\cdot,\cdot\rangle\rightarrow\C$ such that
	\begin{equation}\label{c5}
	  (X^*+X).v=0,\quad\forall X\in\go,\;\forall v\in V.
	\end{equation}
	and the action of $K$ is unitary.
\end{definition}

\begin{remark}
	Since $K$ acts on finite-dimensional spaces, the action is automatically
	unitary on $K_0$. Hence, we have to check that the action is unitary
	only for some representatives of connected components. Since the group
	$G=SU(2,1)$ is connected, it remains to check only \eqref{c5}.
\end{remark}

Now, let us construct the inner product mentioned in Definition \ref{defunit}.
Using the same way of reasoning as we did for \eqref{a15}, one can conclude
that
\begin{equation}\nonumber
  \langle v_{nm}^k,v_{rs}^l\rangle=0,\quad
  v_{nm}^k\in V_{nm},\;v_{rs}^l\in V_{rs}
\end{equation}
for $s\neq m$ or $r-2l\neq n-2k$. It remains to consider the situation
when $s=m$ and $r-2l=n-2k$.

\begin{lemma}\label{lemort}
	Let $V$ be a unitary $(\g,K)$ module,
	$v_{nm}^k\in V_{nm}$ and $v_{n-2k+2l\,m}^{l}\in V_{n-2k+2l\,m}$. Then
	\begin{equation}\nonumber
	  \langle v_{nm}^k,v_{n-2k+2l\,m}^{l}\rangle\neq0
	\end{equation}
	if and only if $k=l$.
\end{lemma}
\begin{proof}
Let us assume that $l>k$,
\begin{equation}\label{c10}
  \langle v_{nm}^k,v_{n-2k+2l\,m}^{l}\rangle\neq0
\end{equation}
and $k$ is the smallest possible with that property. Now,
\begin{equation}\nonumber
  \langle A_{\alpha}.v_{nm}^k,v_{n-2k+2l\,m}^{l-1}\rangle=
  \langle -(k-1)v_{nm}^{k-1}+(n-k)v_{nm}^{k+1},v_{n-2k+2l\,m}^{l-1}\rangle=0.
\end{equation}
By \eqref{c5},
\begin{equation}\nonumber
  \langle v_{nm}^k,A_{\alpha}^*.v_{n-2k+2l\,m}^{l-1}\rangle=
  \langle v_{nm}^k,(l-2)v_{n-2k+2l\,m}^{l-2}+
    (n-2k+l+1)v_{n-2k+2l\,m}^{l}\rangle\neq0.
\end{equation}
It gives a contradiction to assumption \eqref{c10}.
\end{proof}

Let us assume that unitary $(\g,K)$ module $V$ is given.
We want to find relationship among coefficients
$a_{nm}$, $b_{nm}$, $c_{nm}$ and $d_{nm}$. Relation
\begin{equation}\nonumber
  \langle A_{\alpha+\beta}.v_{n-1\,m-3}^k,v_{nm}^k\rangle=
  \langle v_{n-1\,m-3}^k,A_{\alpha+\beta}^*.v_{nm}^k\rangle,
\end{equation}
produces
\begin{equation}\label{c15}
  a_{n-1\,m-3}||v_{nm}^k||^2=
  -\frac{n-k}{n-1}\overline{d_{nm}}||v_{n-1\,m-3}^k||^2.
\end{equation}
Calculation is straightforward. Operator $A_{\alpha+\beta}$ is given by
\eqref{b5}, \eqref{b7} and \eqref{b9}, $A_{\alpha+\beta}^*$ by \eqref{c5} and
Lemma \ref{lemort} is used for both sides. Relation \eqref{c15}, at first
glance, does not look good. Namely, it gives a relationship between
$a_{n-1\,m-3}$ and $d_{nm}$, but it depends on $k$. Let us write it for
$k=1$. It transforms to
\begin{equation}\label{c20}
  a_{n-1\,m-3}||v_{nm}^1||^2=-\overline{d_{nm}}||v_{n-1\,m-3}^1||^2.
\end{equation}
One can apply \eqref{a25} on both sides and obtain
\begin{equation}\nonumber
  a_{n-1\,m-3}\binom{n-1}{k-1}||v_{nm}^k||^2=
  -\overline{d_{nm}}\binom{n-2}{k-1}||v_{n-1\,m-3}^k||^2
\end{equation}
and it is easy to recognize \eqref{c15}. Hence, it is enough to consider
\eqref{c20}. Relation
\begin{equation}\nonumber
  \langle A_{\alpha+\beta}.v_{n+1\,m-3}^{k+1},v_{nm}^k\rangle=
  \langle v_{n+1\,m-3}^{k+1},A_{\alpha+\beta}^*.v_{nm}^k\rangle,
\end{equation}
produces
\begin{equation}\nonumber
  \frac{k}{n}c_{n-1\,m-3}||v_{nm}^k||^2=
  -\overline{b_{nm}}||v_{n+1\,m-3}^{k+1}||^2.
\end{equation}
Using \eqref{a23} it transforms to
\begin{equation}\nonumber
  \frac{k}{n}c_{n+1\,m-3}||v_{nm}^k||^2=
  -\overline{b_{nm}}\frac{k}{n+1-k}||v_{n+1\,m-3}^{k}||^2.
\end{equation}
and
\begin{equation}\nonumber
  c_{n+1\,m-3}||v_{nm}^k||^2=
  -\overline{b_{nm}}\frac{n}{n+1-k}||v_{n+1\,m-3}^{k}||^2.
\end{equation}
Again, it is enough to consider this expression for $k=1$,
\begin{equation}\label{c25}
  c_{n+1\,m-3}||v_{nm}^1||^2=
  -\overline{b_{nm}}||v_{n+1\,m-3}^{1}||^2.
\end{equation}
We continue and consider relations
\begin{equation}\nonumber
  \langle A_{\alpha+\beta}.v_{n-1\,m+3}^{k-1},v_{nm}^k\rangle=
  \langle v_{n-1\,m+3}^{k-1},A_{\alpha+\beta}^*.v_{nm}^k\rangle,
\end{equation}
and
\begin{equation}\nonumber
  \langle A_{\alpha+\beta}.v_{n+1\,m+3}^{k},v_{nm}^k\rangle=
  \langle v_{n+1\,m+3}^{k},A_{\alpha+\beta}^*.v_{nm}^k\rangle.
\end{equation}
However, they do not produce new relation among coefficients $a_{nm}$,
$b_{nm}$, $c_{nm}$ and $d_{nm}$. The same procedure can be repeated for
$B_{\alpha+\beta}$, $A_{\beta}$ and $B_{\beta}$ but is will not produce
new relation. Hence, only \eqref{c20} and \eqref{c25} have to be satisfied.

Now, let us go in opposite direction. We want to find conditions on
coefficients $a_{nm}$, $b_{nm}$, $c_{nm}$ and $d_{nm}$ which will lead us to
unitary (irreducible) $(\g,K)$ module $V$. Lemma \ref{lemort} shows that the
inner product on $V$ such that $K$ acts by unitary operators is given by
expressions $||v_{nm}^1||^2$.
Relation \eqref{c20}, for irreducible $V$, shows that
\begin{equation}\label{c30}
  a_{nm}d_{n+1\,m+3}\in(-\infty,0)
\end{equation}
and \eqref{c25} shows that
\begin{equation}\label{c35}
  b_{nm}c_{n+1\,m-3}\in(-\infty,0).
\end{equation}
One should compare \eqref{c30} and \eqref{c35} with \eqref{a20}.
We claim that it is enough to satisfy \eqref{c30} and \eqref{c35}.
Hence, one has to define the "norm"
for each $V_{nm}$ such that \eqref{c20} and \eqref{c25} are satisfied.
The expression "norm" of $V_{nm}$, by \eqref{a25}, means $||v_{nm}^1||^2$.
It is enough to start with any $K$ type $V_{nm}$ and then define
norms of other $K$ types using \eqref{c20} and \eqref{c25}.
We have to prove that this construction is good. Let us assume that
the norm of $V_{nm}$ is given. We have to show that norms of
$V_{n+2\,m}$, $V_{n\,m+6}$, $V_{n\,m-6}$ and $V_{n-2\,m}$ are well defined.
The norm of $V_{n+2\,m}$ can be calculated in two different ways.
Using \eqref{c25} and \eqref{c20}, one obtains
\begin{equation}\nonumber
  ||v_{n+2\,m}^1||^2=
  -\frac{c_{n+2\,m}}{\overline{b_{n+1\,m+3}}}||v_{n+1\,m+3}^1||^2=
  \frac{c_{n+2\,m}}{\overline{b_{n+1\,m+3}}}\cdot
  \frac{\overline{d_{n+1\,m+3}}}{a_{nm}}||v_{nm}^1||^2.
\end{equation}
Similarly,
\begin{equation}\nonumber
  ||v_{n+2\,m}^1||^2=
  -\frac{\overline{d_{n+2\,m}}}{a_{n+1\,m-3}}||v_{n+1\,m-3}^1||^2=
  \frac{\overline{d_{n+2\,m}}}{a_{n+1\,m-3}}\cdot
  \frac{c_{n+1\,m-3}}{\overline{b_{nm}}}||v_{nm}^1||^2.
\end{equation}
It remains to show that
\begin{equation}\nonumber
  \frac{c_{n+2\,m}}{\overline{b_{n+1\,m+3}}}\cdot
  \frac{\overline{d_{n+1\,m+3}}}{a_{nm}}=
  \frac{\overline{d_{n+2\,m}}}{a_{n+1\,m-3}}\cdot
  \frac{c_{n+1\,m-3}}{\overline{b_{nm}}}.
\end{equation}
It follows from \eqref{b40} and \eqref{b45}.
The norm of $V_{n\,m+6}$ can be also calculated in two different ways.
Using \eqref{c25} and \eqref{c20}, one obtains
\begin{equation}\nonumber
  ||v_{n\,m+6}^1||^2=
  -\frac{\overline{b_{n\,m+6}}}{c_{n+1\,m+3}}||v_{n+1\,m+3}^1||^2=
  \frac{\overline{b_{n\,m+6}}}{c_{n+1\,m+3}}\cdot
  \frac{\overline{d_{n+1\,m+3}}}{a_{nm}}||v_{nm}^1||^2.
\end{equation}
Similarly,
\begin{equation}\nonumber
  ||v_{n\,m+6}^1||^2=
  -\frac{\overline{d_{n\,m+6}}}{a_{n-1\,m+3}}||v_{n-1\,m+3}^1||^2=
  \frac{\overline{d_{n\,m+6}}}{a_{n-1\,m+3}}\cdot
  \frac{\overline{b_{n-1\,m+3}}}{c_{nm}}||v_{nm}^1||^2.
\end{equation}
It remains to show that
\begin{equation}\nonumber
  \frac{\overline{b_{n\,m+6}}}{c_{n+1\,m+3}}\cdot
  \frac{\overline{d_{n+1\,m+3}}}{a_{nm}}=
  \frac{\overline{d_{n\,m+6}}}{a_{n-1\,m+3}}\cdot
  \frac{\overline{b_{n-1\,m+3}}}{c_{nm}}.
\end{equation}
It follows again from \eqref{b40} and \eqref{b45}. Remaining two cases can
be shown similarly.

It remains to notice that we have shown that
\begin{equation}\nonumber
  \langle C.w,v_{nm}^1\rangle=\langle w,(-C).v_{nm}^1\rangle
\end{equation}
for $C=A_{\beta}$, $B_{\beta}$, $A_{\alpha+\beta}$ and $B_{\alpha+\beta}$ and
$w=v_{n\pm1,m\pm3}^1$ if \eqref{c30} and \eqref{c35} are satisfied.
By Lemma \ref{lemort}, it is enough since inner product is equal to 0
in all other cases. Hence operators $A_{\beta}$ ,$B_{\beta}$,
$A_{\alpha+\beta}$ and $B_{\alpha+\beta}$ are unitary. We have proved

\begin{theorem}\label{tmun}
	$(\g,K)$ module $V$ is unitary if and only if \eqref{c30} and \eqref{c35}
	are satisfied.
\end{theorem}

\begin{remark}
	One should compare statements of Theorem \ref{tmun} and \eqref{a20}.
	Each time unitary action is obtained if certain products are real
	negative numbers. We hope that similar statement will be valid for
	some other real reductive groups.
\end{remark}

Now, we want to apply Theorem \ref{tmun} on Theorem \ref{tmgk} and find
all irreducible unitary $(\g,K)$ modules. Firstly, we will analyze modules
$V(c,2t)$ and concentrate on the component which contains $V_{1\,2t}$.
In the second step we will analyze modules $W(r,s)$ for $r>1$.
Theorem \ref{tmun} says that expressions given by \eqref{b75} and \eqref{b80}
have to be negative. Since $\ds \frac{p+1}{p+q+2}>0$ and
$\ds \frac{q+1}{p+q+2}>0$ it reduces to
\begin{equation}\label{c40}
  2c-(p+1)t-p(p+2)<0
\end{equation}
and
\begin{equation}\label{c45}
  2c+(q+1)t-q(q+2)<0.
\end{equation}
We want to find all values of $c$ such that \eqref{c40} and \eqref{c45} are
satisfied.
If $t\geq0$ then it is enough to consider only \eqref{c45}. If $t<0$ then it is
enough to consider only \eqref{c40}. The symmetry shows that it is enough to
consider only one case. Hence, we will consider the case when $t\geq0$.
Since our expression is a quadratic polynomial in variable $q$ we will
consider situations when $t=0$ and $t=1$ separately (when the $x$ coordinate
of the vertex of the parabola is negative). Since $q\in\N$, we will consider
situations $t=2k$, $k\in\N$ and $t=2k+1$, $k\in\N$ separately.

When $t=0$, \eqref{c45} reduces to $-q^2-2q+2c<0$ and it is fulfilled for
$c<0$. Let us define
\begin{equation}\nonumber
  c(0)=0.
\end{equation}
Hence $V(c,0)$ is irreducible unitary for $c\in(-\infty,c(0))$. For $c=0$,
$V(c(0),0)$ is reducible. Let us denote by $U(0)$ irreducible submodule
which contains $V_{10}$. Then $U(0)=V_{10}$ is one-dimensional module.

When $t=1$, \eqref{c45} transforms to $-q^2-q+2c+1<0$ and we set
\begin{equation}\nonumber
  c(1)=-\frac12.
\end{equation}
Hence, $V(c,2)$ is irreducible unitary for $c\in(-\infty,c(1))$ and
$V(c(1),2)$ is reducible. Let $U(2)$ be an irreducible submodule which
contains $V_{12}$. Then $U(2)$ contains $K$ types of the form
\begin{equation}\nonumber
  \{V_{1+p\,2+3p}\,|\,p\in\N\cup\{0\}\}.
\end{equation}

When $t=2k$ for $k\in\N$, \eqref{c45} shows that
\begin{equation}\nonumber
  c(2k)=-\frac{k^2+1}{2}
\end{equation}
and $V(c,4k)$ is irreducible unitary for $c\in(-\infty,c(2k))$. It is also
possible that $V(c,4k)$ has an irreducible unitary submodule which contains
$V_{1\,4k}$. Relation \eqref{c45} transforms to
\begin{equation}\nonumber
  2c+(q+1)t-q(q+2)=-(q-l)(q-(2(k-1)-l))
\end{equation}
for $l\in\{0,\ldots,k-1\}$ and it produces
\begin{equation}\nonumber
  c(l,2k)=\frac{l^2-2(k-1)l-2k}{2}.
\end{equation}
The submodule of $V(c(l,t),2t)$ which contains $V_{1\,4k}$
for $l\in\{0,\ldots,k-1\}$ will be denoted by $U(l,2t)$. It contains
$K$ types of the form
\begin{equation}\label{c50}
  \{V_{1+p+q\,2t+3p-3q}\,|\,p\in\N\cup\{0\},q\in\{0,\ldots,l\}\}.
\end{equation}

When $t=2k+1$ for $k\in\N$, \eqref{c45} shows that
\begin{equation}\nonumber
  c(2k+1)=-\frac{k^2+k+1}{2}
\end{equation}
and $V(c,2(2k+1))$ is irreducible unitary for $c\in(-\infty,c(2k+1))$.
Similarly as in a previous case, it is possible to find irreducible
unitary submodules of $V(c,2(2k+1))$ for some $c$. Again, \eqref{c45}
transforms to
\begin{equation}\nonumber
  2c+(q+1)t-q(q+2)=-(q-l)(q-(2k-1-l))
\end{equation}
for $l\in\{0,\ldots,k-1\}$ and it produces
\begin{equation}\nonumber
  c(l,2k+1)=\frac{l^2-(2k-1)l-2k-1}{2}.
\end{equation}
The submodule of $V(c(l,t),2t)$ which contains $V_{1\,2(2k+1)}$
for $l\in\{0,\ldots,k-1\}$ will be denoted by $U(l,2t)$ and it contains
$K$ types of the form \eqref{c50}.

The situation is very similar for $t<0$. One has to use
\eqref{c40} instead of \eqref{c45}. It is easy to see that
$c(t)=c(-t)$ and $c(l,t)=c(l,-t)$ for
$l\in\{0,\ldots,\left\lfloor\frac{-t}{2}\right\rfloor-1\}$.
The set of $K$ types of $U(-2)$ is
$\{V_{1+q\,-2+3q}\,|\,q\in\N\cup\{0\}\}$
and the set of $K$ types of $U(l,2t)$ is
\begin{equation}\nonumber
\{V_{1+p+q\,2t+3p-3q}\,|\,p\in\{0,\ldots,l\},q\in\N\cup\{0\}\}
\end{equation}
for $l\in\{0,\ldots,\left\lfloor\frac{-t}{2}\right\rfloor-1\}$.

Now, let us consider modules $W(r,s)$ for $r>1$.
Theorem \ref{tmgk} claims that $W(r,s)$ is a submodule of some $V(c,2t)$.
Hence, \eqref{b75} and \eqref{b80} can be applied. Since $r>1$,
expressions given in \eqref{b75} and \eqref{b80} are equal to 0 in the
previous step. Hence, it is enough to find when these two expressions
are equal or less to 0.

We have mentioned that the system given by \eqref{b20} and \eqref{b25}
in variables $a_{rs}d_{r+1\,s+3}$ and $b_{rs}c_{r+1\,s-3}$
($c_{rs}b_{r-1\,s+3}=0$ and $d_{rs}a_{r-1\,s-3}=0$) has a unique solution
\begin{equation}\nonumber
  a_{rs}d_{r+1\,s+3}=-\frac{r(s+r+1)}{2(r+1)}\quad\mbox{and}\quad
  b_{rs}c_{r+1\,s-3}=\frac{r(s-r-1)}{2(r+1)}.
\end{equation}
Since, $r>0$ and $r+1>0$ it remains to consider $s+r+1\geq0$ and
$s-r-1\leq0$. It can not happen that both expressions are equal to 0.
If $s+r+1=0$ then $s-r-1<0$, $W(r,s)$ is reducible and the submodule which
contains $V_{rs}$ will be denoted by $Z(s)$ where $s\in-(\N\setminus\{1\})$.
The set of $K$ types of $Z(s)$ is
\begin{equation}\nonumber
  \{V_{-s-1+q\,s-3q}\,|\,q\in\N\cup\{0\}\}.
\end{equation}
If $s-r-1=0$ then $s+r+1>0$, $W(r,s)$ is reducible and the submodule which
contains $V_{rs}$ will be denoted by $Z(s)$ where $s\in\N\setminus\{1\}$.
The set of $K$ types of $Z(s)$ is
\begin{equation}\nonumber
  \{V_{s-1+p\,s+3p}\,|\,p\in\N\cup\{0\}\}.
\end{equation}
Finally, if $s+r+1>0$ and $s-r-1<0$ then $W(r,s)$ is unitary irreducible
and the set $K$ types is
\begin{equation}\nonumber
  \{V_{r+p+q\,s+p+q}\,|\,p,q\in\N\cup\{0\}\}.
\end{equation}
Hence, we have proved

\begin{theorem}\label{tmunit}
	Irreducible unitary $(\g,K)$ modules are parametrized as follows
	\begin{enumerate}
		\item $V(c,2t)$, $t\in\Z$, $c\in(-\infty,c(t))$,
		\item $U(0)$, $U(2)$, $U(-2)$ and $U(l,2t)$, $t\in\Z\setminus\{0,1,
		  -1\}$, $l\in\{0,\ldots,\left\lfloor\frac{|t|}{2}\right\rfloor-1\}$,
		\item $W(r,s)$, $s+r+1>0$ and $s-r-1<0$,
		\item $Z(s)$, $s\in\Z\setminus\{-1,0,1\}$.
	\end{enumerate}
    All these $(\g,K)$ modules are nonequivalent.
\end{theorem}

One can compare Theorem \ref{tmunit} and results in \cite{hrk72}. Irreducible
unitary representations of $SU(2,1)$ are given on the page 185 by the expression
\begin{equation}\nonumber
    W'=A\cup B_+'\cup B_-'\cup C_+'\cup C_-'\cup D'\cup E_+\cup E_-\cup F.
\end{equation}
Here $A=\{t\in\R\,|\,t>0\}$ (defined on the page 182 for the universal covering
group $\widetilde{SU(2,1)}$ of $SU(2,1)$) stands for irreducible unitary
representations with the spectrum
$\Gamma_0=\{(\frac r2+\frac s2,r-s)\,|\,r,s\in\N\cup\{0\}\}$, defined in the
Proposition 1, case 8 and discussed again on the page 182. The pair
$(\frac r2+\frac s2,r-s)$ denotes the $K$ type of dimension $r+s+1$. We have the
similar notation. The set $A$ corresponds to the set $\{V(c,0)\,|\,c<c(0)=0\}$
which is mentioned in the first part of the Theorem \ref{tmunit}. We can
continue with $B_+'$, but it is clear that the correspondence is not simple.

\section*{Acknowledgement}

This work was supported by the QuantiXLie Centre of Excellence, a project
co financed by the Croatian Government and European Union through the
European Regional Development Fund - the Competitiveness and Cohesion
Operational Programme (Grant KK.01.1.1.01.0004).

\bibliographystyle{alpha} 
\bibliography{b}

\begin{thebibliography}{Kna96}

\bibitem[Bal97]{bal}
M.~Welleda Baldoni.
\newblock General representation theory of real reductive {L}ie groups.
\newblock In {\em Representation theory and automorphic forms ({E}dinburgh,
  1996)}, volume~61 of {\em Proc. Sympos. Pure Math.}, pages 61--72. Amer.
  Math. Soc., Providence, RI, 1997.

\bibitem[Bar47]{bar}
V.~Bargmann.
\newblock Irreducible unitary representations of the {L}orentz group.
\newblock {\em Ann. of Math. (2)}, 48:568--640, 1947.

\bibitem[Kna96]{kn}
Anthony~W. Knapp.
\newblock {\em Lie groups beyond an introduction}, volume 140 of {\em Progress
  in Mathematics}.
\newblock Birkh\"auser Boston, Inc., Boston, MA, 1996.

\bibitem[Kra72]{hrk72}
Hrvoje Kraljevi\'{c}.
\newblock The dual space of the group {${\rm SU}(2,\,1)$} and of its universal
  covering group.
\newblock {\em Glasnik Mat. Ser. III}, 7(27):173--187, 1972.

\bibitem[Kra76]{hrk}
Hrvoje Kraljevi\'c.
\newblock On representations of the group {$SU(n,1)$}.
\newblock {\em Trans. Amer. Math. Soc.}, 221(2):433--448, 1976.

\bibitem[Vog81]{vo}
David~A. Vogan, Jr.
\newblock {\em Representations of real reductive {L}ie groups}, volume~15 of
  {\em Progress in Mathematics}.
\newblock Birkh\"auser, Boston, Mass., 1981.

\end{thebibliography}

\end{document}